\documentclass[a4paper,11pt]{article}
\usepackage[T1]{fontenc}
\usepackage{lmodern,amsmath,amsthm,amsfonts,amssymb,graphicx,float,microtype,thmtools,underscore,mathtools,xurl,thm-restate,wrapfig}
\usepackage[dvipsnames,svgnames,table]{xcolor}
\usepackage[shortlabels]{enumitem}
\setlist[itemize]{topsep=0ex,itemsep=0ex,parsep=0ex}
\setlist[enumerate]{topsep=0ex,itemsep=0ex,parsep=0ex}
\usepackage[unicode=true]{hyperref}
\hypersetup{
colorlinks,
breaklinks=true,
linkcolor={blue!60!black},
citecolor={black},
urlcolor={blue!60!black},
pdftitle={Tree-Decompositions of Graphs with Given Pathwidth}}
\usepackage[capitalise, compress, nameinlink, noabbrev]{cleveref}
\crefname{lem}{Lemma}{Lemmas}
\crefname{thm}{Theorem}{Theorems}
\crefname{cor}{Corollary}{Corollaries}
\newcommand{\defn}[1]{\textcolor{Maroon}{\emph{#1}}}

\usepackage[longnamesfirst,numbers,sort&compress]{natbib}
\makeatletter
\def\NAT@spacechar{~}
\makeatother
\usepackage[tmargin=30mm,bmargin=30mm,lmargin=30mm,rmargin=30mm]{geometry}
\renewcommand{\baselinestretch}{1.1}
\allowdisplaybreaks

\DeclarePairedDelimiter{\floor}{\lfloor}{\rfloor}

\renewcommand{\epsilon}{\varepsilon}
\renewcommand{\emptyset}{\varnothing}

\renewcommand{\geq}{\geqslant}
\renewcommand{\leq}{\leqslant}

\DeclareMathOperator{\dist}{dist}

\DeclareMathOperator{\tw}{tw}

\DeclareMathOperator{\tpw}{tpw}
\DeclareMathOperator{\pw}{pw}

\newcommand{\NN}{\mathbb{N}}

\renewcommand{\thefootnote}{\fnsymbol{footnote}}
\theoremstyle{plain}
\newtheorem{thm}{Theorem}
\newtheorem{lem}[thm]{Lemma}

\crefname{obs}{Observation}{Observations}
\newtheorem*{lem*}{Lemma}
\theoremstyle{definition}

\newtheorem*{conj*}{Conjecture}

\begin{document}
\title{\bf\boldmath\fontsize{18pt}{20pt}\selectfont 
Tree-partitions of graphs with given pathwidth}

\author{David~R.~Wood\,\footnotemark[2]}

\maketitle

\begin{abstract}
Graphs with bounded treewidth and bounded maximum degree are known to have tree-partitions of bounded width. What can be said if the bounded treewidth assumption is strengthened to bounded pathwidth? We prove that every graph with bounded pathwidth and bounded maximum degree has a tree-partition of bounded width, with the extra property that the underlying tree has bounded pathwidth. Moreover, we prove a lower bound showing that the bound on the pathwidth of the underlying tree is within a constant factor of optimal.
\end{abstract}

\footnotetext[2]{School of Mathematics, Monash University, Melbourne, Australia (\texttt{david.wood@monash.edu}). Research supported by the Australian Research Council and by NSERC. }

\renewcommand{\thefootnote}{\arabic{footnote}}

\section{Introduction}

Treewidth\footnote{For a non-empty tree $T$, a \defn{$T$-decomposition} of a graph $G$ is a collection $(B_x\subseteq V(G):x\in V(T))$ of subsets of $V(G)$ (called \defn{bags}) indexed by the nodes of a tree $T$, such that: (a) for every edge $uv\in E(G)$, some bag $B_x$ contains both $u$ and $v$; and (b) for every vertex $v\in V(G)$, the set $\{x\in V(T):v\in B_x\}$ induces a non-empty subtree of $T$. A \defn{tree-decomposition} is a $T$-decomposition for any tree $T$. 
Properties (a) and (b) are respectively called the `edge-property' and `vertex-property' of tree-decompositions. The \defn{width} of a tree-decomposition is the maximum size of a bag, minus $1$. The \defn{treewidth} of a graph $G$, denoted by \defn{$\tw(G)$}, is the minimum width of a tree-decomposition of $G$.} and pathwidth\footnote{A \defn{path-decomposition} is a $P$-decomposition for any path $P$, often simply denoted by the corresponding sequence of bags $(B_1,\dots,B_n)$. The \defn{pathwidth} of a graph $G$, denoted by $\pw(G)$, is the minimum width of a path-decomposition of $G$. By definition, $\tw(G)\leq\pw(G)$ for every graph $G$.}  are graph parameters that respectively measure how similar a given graph is to a tree or a path. They are of fundamental importance in structural and algorithmic graph theory; see \citep{Reed03,HW17,Bodlaender98} for surveys. This paper studies tree-partitions, which are a stronger structure than a tree-decomposition. For a non-empty tree $T$ and graph $G$, a \defn{$T$-partition} of $G$ is a partition $(B_x:x\in V(T))$ of $V(G)$ indexed by $V(T)$ such that for every edge $vw\in E(G)$ if $v\in B_x$ and $w\in B_y$, then $x=y$ or $xy\in E(T)$. The \defn{width} of a $T$-partition is~${\max\{ |{B_x}| : x \in V(T)\}}$. A \defn{tree-partition} is a $T$-partition for any tree $T$. The \defn{tree-partition-width} of $G$, denoted \defn{$\tpw(G)$}, is the minimum width  of a tree-partition of $G$. Tree-partitions were independently introduced by \citet{Seese85} and \citet{Halin91}, and have since been widely investigated \citep{Bodlaender-DMTCS99,BodEng-JAlg97,DO95,DO96,Edenbrandt86, Wood06,Wood09,BGJ22}. Applications of tree-partitions include graph drawing~\citep{CDMW08,GLM05,DMW05,DSW07,WT07}, 
graphs of linear growth~\citep{CDGHHHMW23}, 
nonrepetitive graph colouring~\citep{BW08}, 
clustered graph colouring~\citep{ADOV03,LO18}, 
monadic second-order logic~\citep{KuskeLohrey05}, 
network emulations~\citep{Bodlaender-IPL88, Bodlaender-IC90, BvL-IC86, FF82}, 
size Ramsey numbers~\citep{DKCPS,KLWY21}, 
and the edge-{E}rd{\H{o}}s-{P}{\'o}sa property~\citep {RT17,GKRT16,CRST18}. 
The key property in all these applications is that each vertex appears only once in the tree-partition (unlike in tree-decompositions). 

Tree-partitions are related to tree-decompositions and treewidth, as we now explain. 
Bounded tree-partition-width implies bounded treewidth, as noted by \citet{Seese85}. In particular, for every graph~$G$,
\[\tw(G) \leq 2\tpw(G)-1.\]
For this reason, tree-partition-width has also been called \defn{strong treewidth} \citep{Seese85,BodEng-JAlg97}. Of course, ${\tw(T) = \tpw(T) = 1}$ for every tree~$T$. But in general, $\tpw(G)$ can be much larger than~$\tw(G)$. For example, fan graphs on~$n$ vertices have pathwidth~2 and tree-partition-width~$\Omega(\sqrt{n})$ (see \citep{DO96} and \cref{Fan}). 

\begin{figure}[!ht]
\centering
\includegraphics[width=\textwidth]{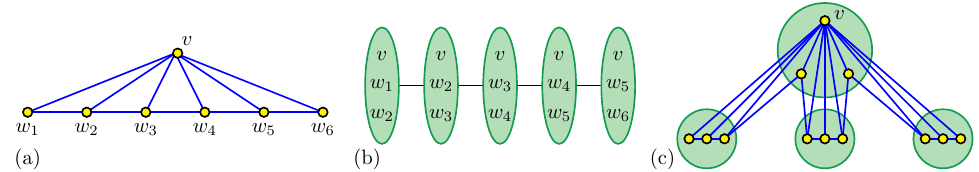}
\caption{(a) Fan graph, (b) path-decomposition with width 2, (c) tree-partition with width~$O(\sqrt{n})$.}
\label{Fan}
\end{figure}

On the other hand, the referee of \citep{DO95} showed that if the maximum degree and treewidth are both bounded, then so is the tree-partition-width, which is one of the most useful results about tree-partitions. A graph $G$ is \defn{non-trivial} if $E(G)\not=\emptyset$. Let \defn{$\Delta(G)$} be the maximum degree of a vertex of $G$. 
 
\begin{thm}[\cite{DO95}]
\label{TreeProduct}
There is a constant $c$ such that for any non-trivial graph $G$,
$$\tpw(G) \leq c (\tw(G)+1)\Delta(G).$$
\end{thm}

\citet{Wood09} showed that \cref{TreeProduct} is best possible up to the value of $c$. The upper bound on $c$ has been steadily improved~\citep{DO95,Wood09,DKKW}, most recently to $8$ by  \citet{DKKW}.

This paper considers the following question: What can be said about tree-partitions of graphs with bounded pathwidth and bounded maximum degree? In particular, can \cref{TreeProduct} be strengthened in this case? For example, one might hope to upper bound the path-partition-width of such graphs. The following lemma\footnote{\cref{PathPartitionDiameter} is roughly equivalent to a known lower bound on the bandwidth of graphs with given diameter~\cite{CS89,CCDG82,Chvatalova81}.} implies this is impossible. Here a \defn{path-partition} is a $P$-partition for any path $P$, and the \defn{path-partition-width} of a graph $G$ is the minimum width of a path-partition of $G$. 

\begin{lem}
\label{PathPartitionDiameter}
Let $G$ be a connected graph with diameter $d$ and path-partition-width $c$. Then $|V(G)|\leq c(d+1)$. 
\end{lem}

\begin{proof}
Let $(D_1,D_2,\dots,D_p)$ be a path-partition of $G$ with width $c$. Let $v\in D_1$ and $w\in D_p$. So $p-1\leq \dist_{G}(v,w)\leq d$, and $|V(G)|\leq cp\leq c(d+1)$.
\end{proof}

To see the relevance of this lemma, consider the example of the \defn{comb} tree $S_n$, which consists 
\begin{wrapfigure}{r}{48mm}
\centering
\vspace*{-1.5ex}
\includegraphics[width=38mm]{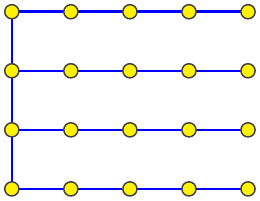}
\vspace*{-1ex}
\caption{The comb $S_4$. \label{Comb}}
\vspace*{-1.5ex}
\end{wrapfigure}
of $n+1$ disjoint paths $P,Q_1,\dots,Q_n$, each with $n$ vertices, where for each $i\in\{1,\dots,n\}$, the $i$-th vertex of $P$ is adjacent to the first vertex of $Q_i$, as illustrated in \cref{Comb}. 
Then $S_n$ has maximum degree 3, pathwidth 2 (by \cref{TreePathwidth} below), and diameter $3n-1$. By \cref{PathPartitionDiameter}, if $S_n$ has path-partition-width $c$, then $n^2=|V(G)|\leq 3cn$, implying $c\geq \frac{n}{3}$. This shows it is impossible to bound the path-partition-width of trees with pathwidth 2 and maximum degree 3. 

Our main result says that graphs with bounded pathwidth and bounded maximum degree have tree-partitions with bounded width, with the extra property that the tree indexing the partition has bounded pathwidth. 

\begin{thm}
\label{main}
 For every graph $G$ with pathwidth $k$ and maximum degree $d\geq 1$, there is a tree $T$ with $\pw(T) \leq 2k+1$ and a $T$-partition of $G$ with width at most $4d(k+1)^2$.
\end{thm}

\cref{main} is proved in \cref{UpperBound}. Our second result shows that the $O(k)$ upper bound on $\pw(T)$ in \cref{main} is optimal, even for trees with maximum degree 3. 

\begin{thm}
\label{LowerBound}
Assume $f$ and $g$ are functions such that for every tree $G$ with pathwidth $k$ and maximum degree $3$, there is a tree $T$ with $\pw(T) \leq g(k)$ and there is a $T$-partition of $G$ with width at most $f(k)$. Then $g(k)\geq \floor{k/2}$.
\end{thm}

\cref{LowerBound} is proved in \cref{sec:LowerBound}.
 
\section{Proof of Upper Bound}
\label{UpperBound}

We consider finite undirected graphs $G$ with vertex-set $V(G)$ and edge-set $E(G)$. For $v\in V(G)$ let $N_G(v):=\{w\in V(G):vw\in E(G)\}$. For $S\subseteq V(G)$, let $N_G(S):=\bigcup\{N_G(v)\setminus S: v\in S\}$. We use standard graph theory notation~\citep{Diestel5}.  

The proof of \cref{main} uses the following folklore lemma. We include the proof for completeness. See \citep{Suderman-IJCGA04,DSW07,EST-IC94,MPS-SJDM85,Scheffler90,Scheffler92} for closely related results.

\begin{lem}
\label{TreePathwidth}
For any integer $k\geq 1$, a tree $T$ has pathwidth at most $k$ if and only if $T$ has a path $P$ such that $\pw(T-V(P)) \leq k-1$.
\end{lem}

\begin{proof}
Let $(B_1,\dots,B_n)$ be a path-decomposition of a tree $T$ with width at most $k$. Since $T$ is connected, there is a path $P$ from $B_1$ to $B_n$. By the properties of tree-decompositions, $P$ has at least one vertex in each of $B_1,\dots,B_n$. So $(B_1\setminus V(P),\dots,B_n\setminus V(P))$ is a path-decomposition of $T-V(P)$ with width at most $k-1$. Thus $\pw(T-V(P)) \leq k-1$. (This argument holds for any connected graph.)\ 

Now assume that a tree $T$ has a path $P$ such that $\pw(T-V(P)) \leq k-1$. Say $P=(v_1,\dots,v_m)$. 
Each component of $T-V(P)$ is adjacent to exactly one vertex in $P$. Let $X_i$ be the union of the components of $T-V(P)$ that are adjacent to $v_i$. 
By assumption, $X_i$ has a path-decomposition $(B^i_1,\dots,B^i_{n_i})$ of width at most $k-1$. Let $C^i_j:=B^i_j\cup\{v_i\}$. 
Thus 
\[ \big( 
C^1_1,\dots,C^1_{n_1},
\{v_1,v_2\},
C^2_1,\dots,C^2_{n_2},
\{v_2,v_3\},
\dots,
\{v_{m-1},v_m\},
C^m_1,\dots,C^m_{n_m}
\big)
\]
is a path-decomposition of $T$ with width $k$. Hence $\pw(T)\leq k$. 
\end{proof}

Let $\NN:=\{0,1,2,\dots\}$. Define the function $f:\NN^3\to\NN$ where 
\[f(k,d,s):=
\begin{cases}
    \max\{s,1\} & \text{ if }k=0\\
    \max\{ s(k+1),\; 2(k+1) \,+\, f(k-1,d,4d(k+1))\} & \text{ if }k\geq 1
\end{cases}
\]

It is easily proved by induction on $k$ that if $s\leq 4d(k+1)$ then $f(k,d,s)\leq 4d(k+1)^2$. So the next lemma implies \cref{main} (taking $S=\emptyset$).

\begin{lem}
\label{MainLemma}
 For every graph $G$ with pathwidth at most $k$ and maximum degree at most $d$, for every set $S$ of vertices in $G$, there is a tree $T$ with $\pw(T)\leq 2k+1$ and there is a $T$-partition $(B_x:x \in V(T))$ of $G$ such that $|B_x| \leq f(k,d,|S|)$  for each $x \in V(T)$, and $S\subseteq B_r$ for some $r\in V(T)$. 
\end{lem}

\begin{proof}
We proceed by induction on $k\geq 0$. First suppose that  $k=0$. So  $E(G)=\emptyset$. 
Say $V(G-S)=\{v_1,v_2,\dots,v_m\}$. 
Then the path-partition $(S,\{v_1\},\{v_2\},\dots,\{v_m\})$ satisfies the claim. 

Now assume that $k\geq 1$ (refer to \cref{ProofFigure}). 
Let $(D_y:y \in V(P))$ be a path-decomposition of $G$ with width at most $k$, where the first and last bags are empty. Let $X$ be a minimal set of vertices in $P$ such that the first and last nodes of $P$ are in $X$, and $S\subseteq \bigcup\{ D_x: x\in X\}$. So $|\bigcup\{ D_x:x\in X\}| \leq |S|(k+1)$. 

\begin{figure}[!ht]
    \includegraphics[width=\textwidth]{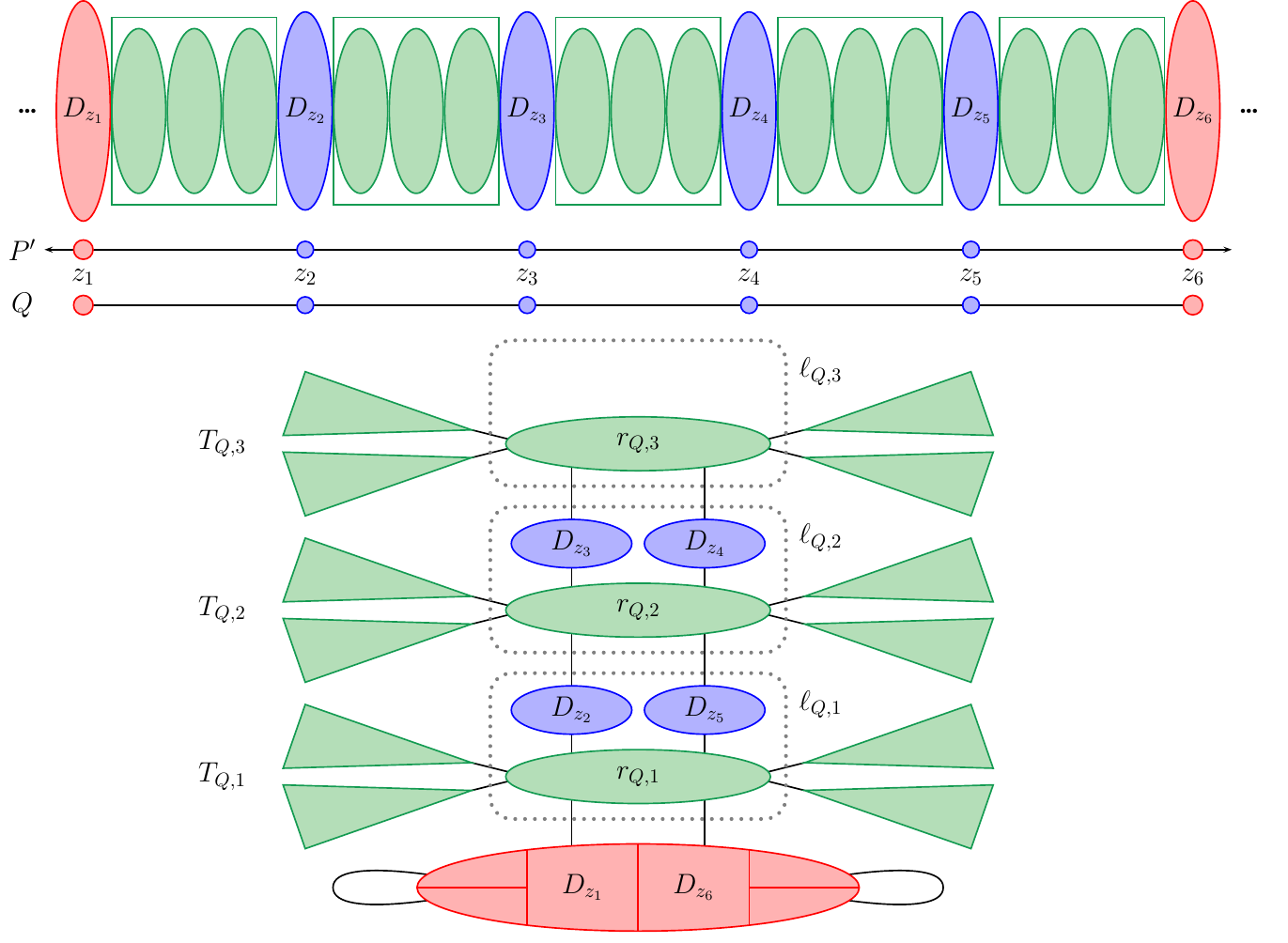}
    \caption{Construction in the proof of \cref{MainLemma}, where $X$ is red, $Y$ is blue.}
    \label{ProofFigure}
\end{figure}

Let $Y$ be a maximal set of vertices in $P-X$ such that $D_x\cap D_y=\emptyset$ for all $x\in X$ and $y\in Y$, and $D_{y_1} \cap D_{y_2}=\emptyset$ for all distinct $y_1,y_2\in Y$. 
Let $Z:= \bigcup\{D_z: z\in X\cup Y\}$. Let $G':= G-Z$. 
For all $a\in V(P-(X\cup Y))$ there exists $z\in X\cup Y$ such that $D_a\cap D_z\neq\emptyset$. 
So  $(D_a-Z: a \in V(P))$ is a path-decomposition of $G'$ with width at most $k-1$, and $\pw(G')\leq k-1$. 

Let $P'$ be the path obtained from $P$ by contracting each component of $P-(X\cup Y)$ to a single edge. So $V(P')=X\cup Y$. 

Consider an edge $e=z_1z_2$ of $P'$. Let $P_e$ be the $z_1z_2$-path in $P$. 
Let \[G_e:= G[ \bigcup\{ D_a\setminus Z:a\in V(P_e-z_1-z_2)\} ].\]
Note that $G_e\subseteq G'$, so $\pw(G_e)\leq k-1$. For any edge $vw$ of $G$ with $v\in V(G_e)$ and $w\not\in V(G_e)$, by the properties of tree-decompositions, $w\in D_{z_1}\cup D_{z_2}$. Let $S_e:= N_G( D_{z_1}\cup D_{z_2}) \cap V(G_e)$. So $|S_e|\leq 2(k+1)d$. Equivalently, $S_e$ is the set of vertices in $G_e$ adjacent to a vertex not in $G_e$. 

For distinct $e,e'\in E(P')$, the subgraphs $G_e$ and $G_{e'}$ are disjoint with no edge between them, since if $z$ is any node of $P'$ between $e$ and $e'$, then $D_z$ separates $G_e$ and $G_{e'}$. 

Say a subpath $Q$ of $P'$ is \defn{key} if the first vertex of $Q$ is in $X$, the last vertex of $Q$ is in $X$, and no internal vertex of $Q$ is in $X$. 
The key subpaths of $P'$ partition the edges of $P'$. 
Consider a key subpath $Q$ of $P'$. 
Let $x_1$ and $x_2$ be the endpoints of $Q$. So $x_1,x_2\in X$. 
For each integer $i\geq 0$, 
let $Y_{Q,i}$ be the set of vertices $y\in V(Q)$ with $\dist_{Q}( \{x_1,x_2\},y)=i$. 
Observe that $|Y_{Q,i}|\leq 2$. 
Let $Q'$ be the path $(\ell_{Q,0},\ell_{Q,1},\ell_{Q,2},\dots,\ell_{Q,m+1})$, where $m$ is the maximum integer such that $Y_{Q,m}\neq\emptyset$.  
Consider each edge $e=y_1y_2$ of $Q$. 
Say $y_1\in Y_{Q,i}$ and $y_2\in Y_{Q,j}$ with $i\leq j$. 
By construction, $j\leq i+1$. 
`Map' $e$ to $\ell_{Q,i+1}$.
Consider a node $\ell_{Q,i}$. 
Let $E_{Q,i}$ be the set of edges of $Q$ mapped to $\ell_{Q,i}$. 
By construction, $|E_{Q,i}|\leq 2$.  
Let $G_{Q,i}:=\bigcup\{G_e:e\in E_{Q,i}\}$, which is a subgraph of $G'$. So $\pw(G_{Q,i})\leq k-1$. 
Let $S_{Q,i}:=\bigcup\{S_e:e\in E_{Q,i}\}$. 
So $|S_{Q,i}|\leq 4(k+1)d$. 
By induction, there is a tree-partition $(B_x:x \in V(T_{q,i}))$ of $G_{Q,i}$ such that $|B_x| \leq f(k-1,d,4(k+1)d)$  for each $x \in V(T_{Q,i})$, and $\pw(T_{Q,i}) \leq 2(k-1)+1=2k-1$, and $S_{Q,i} \subseteq B_{r_{Q,i}}$ for some $r_{Q,i}\in V(T_{Q,i})$. 

We now define the tree-partition $(B_x:x\in V(T))$ of $G$. 
Initialise $T$ to be the tree with one node $\alpha$. 
Let $B_\alpha:=\bigcup\{D_x: x\in X\}$. 
So $S\subseteq B_\alpha$, and $|B_\alpha|\leq |S|(k+1)$. 
For each key subpath $Q$ of $P'$, 
add the path $Q'$ to $T$, where $\ell_{Q,0}$ is identified with $\alpha$. 
For each $i\geq 1$, 
add $T_{Q,i}$ to $T$, where the node $r_{Q,i}$ is identified with the node $\ell_{Q,i}$ of $Q'$. 
Let 
$B_{\ell_{Q,i}} := \bigcup\{ D_y : y \in Y_{Q,i} \} \cup B_{r_{Q,i}} $. 
So $|B_{\ell_{Q,i}}|\leq 2(k+1) + f(k-1,d,4d(k+1))\}$.
Observe that $T$ is a tree. 

We now show that $(B_x:x\in V(T))$ is a tree-partition of $G$. 
By construction, the bags are pairwise disjoint, and each vertex $v$ of $G$ appears in exactly one bag $B_x$. 
Consider an edge $vw$ of $G$ with $v\in B_a$ and $w\in B_b$. Our goal is to show that $a=b$ or $ab\in E(T)$. 
By assumption, there is a node $x\in V(P)$ with $v,w\in D_x$.

Case 1. $v,w\in D_x$ for some node $x\in X$: 
Then $v,w\in B_\alpha$, implying $a=b$, as desired. 
Now assume this case does not occur. 

Case 2. $v,w\in D_y$ for some node $y\in Y$: 
Then $v,w\in B_{\ell_{Q,i}}$ for some $Q$ and $i$, implying $a=b$, as desired. 
Now assume this case does not occur. 

Case 3. $v\in D_{z_1}$ and $w\in D_{z_2}$ for some distinct nodes $z_1,z_2\in X\cup Y$: 
Taking $z_1$ and $z_2$ at minimum distance in $P$, by the properties of tree-decompositions, 
$z_1z_2$ is an edge of some key subpath $Q$ (since Cases 1--2 do not occur). 
So $v\in B_{\ell_{Q,i}}$ and $w\in B_{\ell_{Q,i+1}}$ for some $i\geq 0$. 
Since $\ell_{Q,i}$ is adjacent to $\ell_{Q,i+1}$ in $T$, we have $ab\in E(T)$, as desired. 
Now assume this case does not occur. 

Case 4. $v,w\in D_x$ for some node $x\in  V(P)-(X\cup Y)$: 
Let $e=z_1z_2$ be the edge of $P'$, such that $x$ is between $z_1$ and $z_2$ in $P$. 
So $e$ is mapped to some node $\ell_{Q,i}$. 
Since Cases 1--3 do not occur, without loss of generality, $w\not\in D_{z_1}\cup D_{z_2}$. 
If $v\not\in D_{z_1}\cup D_{z_2}$ then $v,w\in V(G_e)$, and by induction, $a=b$ or $ab\in E(T_{Q,i})\subseteq E(T)$, as desired. 
Otherwise $v\in D_{z_1}\cup D_{z_2}$.
By construction, $w\in S_e$, implying $w\in B_{r_{Q,i}}$, 
and $v\in B_{\ell_{Q,i-1}}\cup B_{\ell_Q{,i+1}}$. 
Since $\ell_{Q,i-1},\ell_{Q,i},\ell_{Q,i+1}$ is a path in $T$, and $r_{Q,i}$ is identified with $\ell_{Q,i}$, 
we have $ab\in E(T)$, as desired. 

Hence $(B_x:x\in V(T))$ is a tree-partition of $G$. 
The width is at most 
$\max\{ |S|(k+1), 2(k+1)+ f(k-1,d,4d(k+1))=f(k,d,|S|)$, as desired. 

It remains to show that $\pw(T)\leq 2k+1$. 
By \cref{TreePathwidth}, it suffices to show that $\pw(T-\alpha)\leq 2k$. 
Since the pathwidth of any graph equals the maximum pathwidth of its connected components, 
it suffices to show that each component of $T-\alpha$ has pathwidth at most $2k$. 
For each component $C$ of $T-\alpha$ there is a path $Q'$ in $C$ such that each component of $C-V(Q')$ is a subtree of $T_{Q,i}$ for some $i$. Since $\pw(T_{Q,i})\leq 2k-1$, we have $\pw(C)\leq 2k$, as desired. 
\end{proof}

Note that there is a straightforward extension of \cref{main}, where instead of starting with a path-decomposition we start with a tree-decomposition indexed by a tree with small pathwidth. In particular, let $T$ be a tree with $\pw(T)\leq\ell$. Let $(B_x:x\in V(P))$ be a path-decomposition of $T$ with width at most $\ell$. Let $G$ be a graph with maximum degree $d$ that has a tree-decomposition $(C_x:x\in V(T))$ of width $k$. Let $B'_x:=\bigcup\{ C_y: y \in B_x\}$ for each $x\in V(P)$. So $(B'_x:x\in V(P))$ is a path-decomposition of $G$ with width at most $(\ell+1)(k+1)-1$. By \cref{main}, $G$ has a $T'$-partition of width $O(\ell^2 k^2d)$ for some tree $T'$ with $\pw(T')\leq 2(\ell+1)(k+1)-1$.

\section{Proof of Lower Bound}
\label{sec:LowerBound}

To prove \cref{LowerBound}, we define the following sequence $G_{1},G_{2},\dots$ of trees with maximum degree at most 3, each rooted at a vertex of degree at most 2, such that $\pw(G_{i})\leq i$, as illustrated in \cref{G4}. Fix a large integer $n$. Let $G_{1}$ be the $n$-vertex path rooted at an endpoint. Assume that $i\geq 2$ and $G_1,\dots,G_{i-1}$ have been defined. Define $G_{i}$ as follows. Let $P$ be an $n$-vertex path. For each vertex $v$ in $P$, add a copy of $G_{i-1}$ to $G_{i}$ whose root vertex is adjacent to $v$. Since the root of $G_{{i-1}}$ has degree at most 2, $G_{i}$ has maximum degree at most 3. Root $G_{i}$ at an endpoint of $P$, which has degree at most 2 in $G_{i}$. Call $P$ the \defn{central path} of $G_{i}$. By \cref{TreePathwidth}, $\pw(G_{i})\leq i$. 

\begin{figure}[!ht]
    \centering
    \includegraphics[scale=0.8]{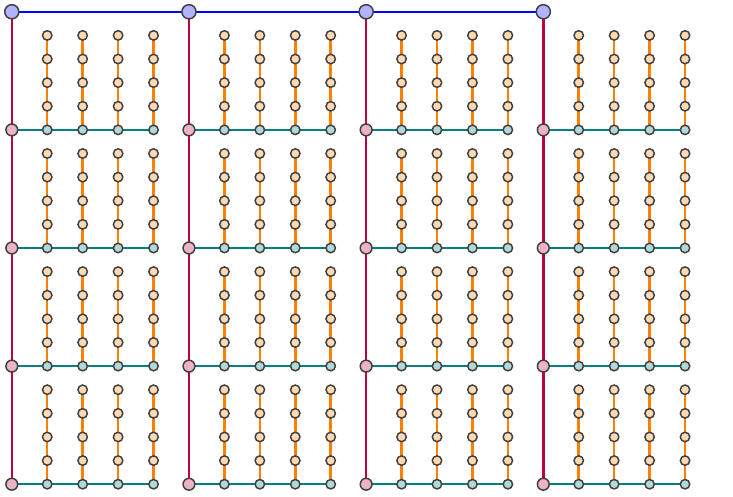}
    \caption{$G_4$ with $n=4$.}
    \label{G4}
\end{figure}

\cref{LowerBound} follows from the next lemma, taking $c:=f(2k)$.

\begin{lem}
For any integers $c,k\geq 1$ if $G_1,\dots,G_{2k}$ are defined with respect to an integer $n>3c(c+1)$, then for any tree $T$, if $G_{2k}$ has a $T$-partition of width at most $c$, then $\pw(T)\geq k$.
\end{lem}

\begin{proof}
We proceed by induction on $k$. The base case $k=1$ is trivial (since $|V(G_2)|\geq c+1$ implies $|V(T)|\geq 2$ and $\pw(T)\geq 1$). 
Now assume that $k\geq 2$ and the result holds for $k-1$ (that is, if $G_{2k-2}$ has a $T$-partition of width at most $c$, then $\pw(T)\geq k-1$).  

Suppose for the sake of contradiction that there is a tree $T$ with $\pw(T)\leq k-1$ and $G_{2k}$ has a $T$-partition $(B_x:x\in V(T))$ of width at most $c$. By \cref{TreePathwidth}, $T$ contains a path $Q$ such that $\pw(T-V(Q))\leq k-2$. Let $R:=\bigcup\{B_x:x\in V(Q)\}$. So $(B_x:x\in V(Q))$ is a path-partition of $G[R]$ with width at most $c$. 

Each component $C$ of $G_{2k}-R$ has a $T'$-partition of width at most $c$, where $T'$ is a component of $T-V(Q)$. So $\pw(T')\leq k-2$. By induction, $C$ contains no copy of $G_{2k-2}$. That is, every copy of $G_{2k-2}$ in $G_{2k}$ has a vertex in $R$. Also, the neighbours of $C$ in $R$ are in one bag $B_x$, so $C$ has at most $c$ neighbours in $R$.

Let $H$ be the subgraph of $G_{2k}$ consisting of the central path of $G_{2k}$ and the central paths of every copy of $G_{2k-1}$. So $H$ is a copy of the comb graph $S_n$. 
Colour a vertex of $H$ `red' if it is in $R$. 
Colour a vertex of $H$ `green' if it is in the central path of $G_{2k}$ but not in $R$. 
Colour a vertex of $H$ `blue' if it is in the central path of a copy of $G_{2k-1}$ but is not in $R$. 
Every vertex of $H$ is assigned one colour. 

Let $H'$ be the graph whose vertices are the red vertices in $H$, where $xy\in E(H')$ whenever every internal vertex of the $xy$-path in $H$ is blue or green. These internal vertices are in a single component $C$ of $G_{k}-R$, with both $x$ and $y$ in $N(C)\cap R$. So $x$ and $y$ are in the same bag $B_x$. Thus $(B_x:x\in V(Q))$ is a path-partition of $H'$. By construction, $H'$ is connected. Moreover, the diameter of $H'$ is at most the diameter of $H$, which is $3n-1$. 

Consider a blue vertex $v$ of $H$. Recall that the copy of $G_{2k-2}$ corresponding to $v$ has a vertex $v'$ in $R$. Choose such a vertex $v'$ at minimum distance from $v$ in $G_{2k}$. So the $vv'$-path in $G_{2k}$ (excluding $v'$) avoids $R$. Say $v_1,\dots,v_k$ is a path of blue vertices in $H$. Then $v_1,\dots,v_k$ is contained in a single component $C$ of $G_{2k}-R$, and $C$ contains the $v_iv_i'$-path (excluding $v_i'$) for each $i$. Thus $v'_1,\dots,v'_k$ are neighbours of $C$. Hence  $v'_1,\dots,v'_k$ are contained in a single bag $B_x$, implying $k\leq c$. So in the central path of each copy of $G_{2k-1}$, at most $c$ consecutive vertices are blue. Thus in each central path of $G_{2k-1}$ at least $n/(c+1)$ vertices are red. Hence, $|V(H')|\geq n^2/(c+1)$. 

By \cref{PathPartitionDiameter},  $n^2/(c+1) \leq |V(H')| \leq 3cn$, implying $n\leq 3c(c+1)$. This is the desired contradiction for $n>3c(c+1)$. 
\end{proof}

We finish with an open problem. What is the optimal width bound in \cref{main}? In particular, are there function $f$ and $g$ such that $f(k,d)\in o(k^2d)$, and for every graph $G$ with pathwidth $k$ and maximum degree $d$, there is a tree $T$ with $\pw(T) \leq g(k)$ and a $T$-partition of $G$ with width at most $f(k,d)$?

\subsection*{Acknowledgements} Thanks to Kevin Hendrey and Freddie Illingworth, with whom the questions motivating the present paper were formulated. 

{
\fontsize{10pt}{11pt}\selectfont
\def\soft#1{\leavevmode\setbox0=\hbox{h}\dimen7=\ht0\advance \dimen7 by-1ex\relax\if t#1\relax\rlap{\raise.6\dimen7 \hbox{\kern.3ex\char'47}}#1\relax\else\if T#1\relax \rlap{\raise.5\dimen7\hbox{\kern1.3ex\char'47}}#1\relax \else\if d#1\relax\rlap{\raise.5\dimen7\hbox{\kern.9ex \char'47}}#1\relax\else\if D#1\relax\rlap{\raise.5\dimen7 \hbox{\kern1.4ex\char'47}}#1\relax\else\if l#1\relax \rlap{\raise.5\dimen7\hbox{\kern.4ex\char'47}}#1\relax \else\if L#1\relax\rlap{\raise.5\dimen7\hbox{\kern.7ex \char'47}}#1\relax\else\message{accent \string\soft \space #1 not defined!}#1\relax\fi\fi\fi\fi\fi\fi}

}
\end{document}